\documentclass[12pt]{article}
\usepackage{amsmath,amssymb,amsfonts}
\usepackage{mathrsfs,mathtools,bm,eufrak}
\usepackage{xcolor}
\usepackage{geometry}
\usepackage[colorlinks=true,
linkcolor=blue,citecolor=blue,
urlcolor=blue]{hyperref}
\usepackage{enumerate}
\makeatletter
\def\@seccntDot{.}
\def\@seccntformat#1{\csname the#1\endcsname\@seccntDot\hskip 0.5em}
\renewcommand\section{\@startsection{section}{1}{\z@}%
{18\p@ \@plus 6\p@ \@minus 3\p@}%
{9\p@ \@plus 6\p@ \@minus 3\p@}%
{\large\bfseries\boldmath}}
\renewcommand\subsection{\@startsection{subsection}{2}{\z@}%
{12\p@ \@plus 6\p@ \@minus 3\p@}%
{3\p@ \@plus 6\p@ \@minus 3\p@}%
{\bfseries\boldmath}}
\renewcommand\subsubsection{\@startsection{subsubsection}{3}{\z@}%
{12\p@ \@plus 6\p@ \@minus 3\p@}%
{\p@}%
{\bfseries\boldmath}}
\makeatother

\usepackage{microtype}

\usepackage[amsmath,thmmarks]{ntheorem}
\theoremstyle{plain}
\newtheorem{theorem}{Theorem}
\newtheorem{lemma}{Lemma}

\theorembodyfont{\normalfont}

\newtheorem{claim}{Claim}

\theoremstyle{nonumberplain}
\theoremseparator{}
\theoremsymbol{\ensuremath{\Box}}
\newtheorem{proof}{\bf Proof.}

\allowdisplaybreaks
\title{\bf On distance-regular Cayley graphs of generalized dicyclic groups}

\author{{Xueyi Huang$^{a, b}$, Kinkar Chandra Das$^{b,}$\footnote{Corresponding author.}\setcounter{footnote}{-1}\footnote{\emph{E-mail address:} huangxymath@163.com (X. Huang), kinkardas2003@gmail.com (K. C. Das).}}\\[2mm]
\small $^a$School of Mathematics, East China University of Science and Technology, \\
\small  Shanghai 200237, P.R. China\\
\small $^b$Department of Mathematics, Sungkyunkwan University,\\
\small Suwon 16419, Republic of Korea
}

\date{}

\begin{document}
\maketitle

\begin{abstract}
Let $G$ be a generalized dicyclic group with identity $1$. An inverse closed  subset $S$ of $G\setminus\{1\}$ is called minimal if $\langle S\rangle=G$ and there exists some $s\in S$ such that $\langle S\setminus\{s,s^{-1}\} \rangle\neq G$. In this paper, we characterize distance-regular Cayley graphs $\mathrm{Cay}(G,S)$ of $G$ under the condition that $S$ is minimal.

\par\vspace{2mm}

\noindent{\bfseries Keywords:} Distance-regular graph, Cayley graph, Generalized dicyclic group
\par\vspace{1mm}

\noindent{\bfseries 2010 MSC:} 05C25
\end{abstract}

\section{Introduction}

Let $G$ be a group with identity $1$, and  let $S$ be an inverse closed subset of $G\setminus\{1\}$. The \textit{Cayley graph}  of $G$ with respect to  $S$, denoted by $\mathrm{Cay}(G,S)$,  is the graph with vertex set $G$ in which  two vertices $g,h\in G$ are adjacent if and only if there exists some $s\in S$ such that $g=hs$. Here $S$ is called the \textit{connection set} of $\mathrm{Cay}(G,S)$. Clearly, $\mathrm{Cay}(G,S)$ is a regular graph which is  connected if and only if  $\langle S\rangle=G$. 
If $\langle S\rangle=G$ and there exists some $s\in S$ such that  $\langle S\setminus\{s,s^{-1}\}\rangle\neq G$, then we say that  $S$ is \textit{minimal} (with respect to $s$). Furthermore, it is easy to see that  the action of $G$ on itself by left multiplication gives an automorphism subgroup of $\mathrm{Cay}(G,S)$. Thus, for any proper subgroup $H$ of $G$ and for any  $g_1,g_2\in G$, the subgraphs of $\mathrm{Cay}(G,S)$ induced by $g_1H$ and $g_2H$ are isomorphic.

Let $\Gamma$ be a connected graph with vertex set $V(\Gamma)$ and edge set $E(\Gamma)$. The length of a shortest path between two vertices $u$ and $v$ of $\Gamma$ is called the \textit{distance} between $u$ and $v$, and denoted by  $\partial_\Gamma(u,v)$. The \textit{diameter} of $\Gamma$ is defined as $d_\Gamma=\max\{\partial_\Gamma(u,v):u,v\in V(\Gamma)\}$. For any $v\in V(\Gamma)$, let $N_i^\Gamma(v)$ denote the set of vertices  in $\Gamma$ which are at distance $i$ from $v$. In particular, we denote $N^\Gamma(v)=N_1^\Gamma(v)$. When $\Gamma$ is clear from the context, we use  $\partial$, $d$, $N_i$ and $N$ instead of  $\partial_\Gamma$, $d_\Gamma$, $N_i^\Gamma$ and $N^\Gamma$, respectively.

Let  $\Gamma$ be a connected graph with diameter $d$. For any $u,v\in V(\Gamma)$ with $\partial(u,v)=i$ ($0\leq i\leq d$), we denote
$$
c_i(u,v)=|N_{i-1}(u)\cap N(v)|,~~ a_i(u,v)=|N_{i}(u)\cap N(v)|, ~~ b_i(u,v)=|N_{i+1}(u)\cap N(v)|.
$$
Here we take $c_0(u,v)=b_d(u,v)=0$. If $c_i(u,v)$, $b_i(u,v)$ and $a_i(u,v)$ do not depend on the choice of $u,v$ with $\partial(u,v)=i$ (that is, depend only on the distance $i$ between $u$ and $v$) for all $0\leq i\leq d$, then we say that $\Gamma$ is a \textit{distance-regular graph}.

For a distance-regular graph $\Gamma$ with diameter $d$, we set  $c_i=c_i(u,v)$, $a_i=a_i(u,v)$ and $b_i=b_i(u,v)$, where $u,v\in V(\Gamma)$ with $\partial(u,v)=i$. Clearly, $\Gamma$ is a regular graph with valency $k=b_0$, and
$a_i+b_i+c_i=k$ for $0\leq i\leq d$.  The numbers $a_i$, $b_i$, $c_i$ ($0\leq i\leq d$) are called the \textit{intersection numbers} of $\Gamma$. Note that, in the graph $\Gamma$,  every pair of adjacent vertices have $a_1$ common neighbors, and every pair of vertices at distance $2$ have $c_2$ common neighbors. In particular, if $d=2$ then $\Gamma$ is also called a \textit{strongly regular graph}, that is, a connected regular graph such that the number of common neighbors of two distinct vertices depends only on whether these vertices are adjacent or not. 

As the generalization of  distance-transitive graphs, the concept of distance-regular graphs was introduced by Biggs  (see the monograph \cite{B74}  from 1974). In the past half century, it was found  that distance-regular graphs not only have many important applications in design theory and coding theory, but are also closely related to some other subjects, such as finite group theory, representation theory, and association schemes. For more detailed results on distance-regular graphs, we refer the reader to the famous monograph by Brouwer, Cohen and Neumaier \cite{BCN89}, and the nice survey paper by van Dam, Koolen and Tanaka \cite{DKT16}.

The research of distance-regular Cayley graphs originated from the investigation of regular partial difference sets. Such kinds of sets  are actually equivalent to strongly regular Cayley graphs \cite{M94}. A classic work on this topic is a characterization of strongly regular Cayley graphs of cyclic groups, which was achieved by Bridges and Mena \cite{BM79}, Ma \cite{M84}, and partially by Maru\v{s}i\v{c} \cite{M89}. Also,  strongly regular Cayley graphs of $\mathbb{Z}_{p^n}\times \mathbb{Z}_{p^n}$ were classified by Leifman and Muzychuck \cite{LM05}. However, as we know, strongly regular Cayley graphs of general groups, even for abelian groups, are  far from being completely classified.

With regard to  distance-regular  graphs, Miklavi\v{c} and Poto\v{c}nik \cite{MP03,MP07} (almost) classified  distance-regular Cayley graphs of cyclic groups and dihedral groups. Miklavi\v{c} and \v{S}parl \cite{MS14,MS20} characterized  distance-regular Cayley graphs of abelian groups and generalized dihedral groups under the condition that the corresponding  connection set is minimal. Abdollahi, van Dam and Jazaeri \cite{ADJ17} determined distance-regular Cayley graphs of diameter at most three with least eigenvalue $-2$. Very recently, van Dam and Jazaeri \cite{DJ19} determined the distance-regular Cayley graphs with valency at most $4$, the Cayley graphs among the distance-regular graphs with known putative intersection arrays for valency $5$, and the Cayley graphs among all distance-regular graphs with girth $3$ and valency $6$ or $7$.  In addition, they also studied bipartite distance-regular Cayley graphs with diameter  $3$ or $4$ \cite{DJ22}.

In this paper, inspired by the work of Miklavi\v{c} and \v{S}parl \cite{MS14,MS20}, we focus on characterizing distance-regular Cayley graphs of generalized dicyclic groups under the condition that the corresponding connection set is minimal.  Let $A$ be an abelian group of order $2n$ ($n>1$) with exactly one involution $\alpha$, and let $G$ be the \textit{generalized dicyclic group} generated by $A$ and $t$ where $t^2=\alpha$ and $t^{-1}xt=x^{-1}$ for all $x\in A$ (see \cite[p. 229]{M94} or \cite[p. 392]{S57}). 
Clearly, $G$ is a non-abelian group of order $4n$, and $\alpha$ is the unique element of order $2$ in $G$. The main result is as follows.

\begin{theorem}\label{thm::main}
Let $G$ be a generalized dicyclic group, and let $S$ be an inverse closed subset of $G\setminus\{1\}$ which generates $G$ and for which there exists some $s\in S$ such that $\langle S\setminus \{s,s^{-1}\}\rangle\neq G$. Then  $\mathrm{Cay}(G,S)$ is distance-regular if and only if it is isomorphic to $K_{4,4}$, the complete bipartite graph on eight vertices with two parts of equal size.
\end{theorem}

\section{Proof of Theorem \ref{thm::main}}

In this section, we give the proof of Theorem \ref{thm::main}. Before doing this, we need a result regarding the characterization of distance-regular graphs which can be decomposed into the Cartesian products of two smaller graphs.

Given two graphs $\Gamma_1$ and $\Gamma_2$, the \textit{Cartesian product} $\Gamma_1\square \Gamma_2$ is the graph with vertex set $V(\Gamma_1)\times V(\Gamma_2)$ in which two vertices $(u_1,v_1),(u_2,v_2)\in V(\Gamma_1)\times V(\Gamma_2)$ are adjacent if and only if $u_1=u_2$ and $v_1,v_2$ are adjacent in $\Gamma_2$, or $v_1=v_2$ and $u_1,u_2$ are adjacent in $\Gamma_1$. It is known that if  a nontrivial connected graph is a Cartesian product, it can be factorized uniquely as a Cartesian product of prime factors, graphs that cannot themselves be decomposed as Cartesian products of graphs \cite{S60,V63}.

For positive integers $d$ and $q$, the \textit{Hamming graph} $H(d,q)$ is  the Cartesian product of $d$ copies of the complete graph $K_q$.  For a nonnegative integer $n$ and  a positive integer $m$, the \textit{Doob graph} $D(n,m)$ is the Cartesian product of $H(n,4)$ with $m$ copies of the Shrikhande graph $\mathrm{Cay}(\mathbb{Z}_4\times \mathbb{Z}_4,\{\pm (1,0),\pm (0,1),\pm (1,1)\})$. Here for the case $n=0$ we just take the Cartesian product of $m$ copies of the Shrikhande graph. It is known that both $H(d,q)$ and $D(n,m)$ are distance-regular graphs (cf. \cite{BCN89}).

In \cite{MS20},  Miklavi\v{c} and \v{S}parl  gave the following characterization for distance-regular Cartesian products based on a result of Stevanovi\'{c} \cite{S04}.

\begin{lemma}[\cite{MS20}]\label{lem::CP}
Let $\Gamma=\Gamma_1\square \Gamma_2$, where $\Gamma_1$ and $\Gamma_2$ are nontrivial graphs. If $\Gamma$ is distance-regular, then $\Gamma$ is isomorphic to a Hamming graph $H(d,q)$ or to a Doob graph $D(n,m)$.
\end{lemma}

Let $G$ be a group. For any $g\in G$, we denote by $o(g)$ the order of $g$, and for any subgroup $H\leq G$, we denote by $[G:H]$ the index of $H$ in $G$, i.e., the number of left cosets of $H$ in $G$.

Now we begin to prove Theorem \ref{thm::main}. For the sake of convenience, we keep the following notation for the remaining part of this section.
{\flushleft\bf Notation.} Let $A$ be an abelian group of order $2n$ ($n>1$) with exactly one involution $\alpha$, and let $G$ be the generalized dicyclic group generated by $A$ and $t$ where $t^2=\alpha$ and $t^{-1}xt=x^{-1}$ for all $x\in A$. Let $S$ be an inverse-closed subset of $G\setminus\{1\}$ with $\langle S\rangle=G$  for which there exists $s\in S$ such that $H=\langle S\setminus\{s,s^{-1}\} \rangle$ is a proper subgroup of $G$. Assume that $\Gamma=\mathrm{Cay}(G,S)$ is distance-regular of diameter $d$ with intersection numbers $a_i$, $b_i$, $c_i$ ($0\leq i\leq d$). Let $\Gamma'=\mathrm{Cay}(H,S\setminus\{s,s^{-1}\})$.

As $n>1$, it is clear that $\{s,s^{-1}\}$ cannot generate $G$.  Thus $|S|\geq 3$. Also note that $G=A\cup tA$.  We divide  our discussion into the following  two parts.

\subsection{The case $s\in tA$}

Note that for any  $x,y\in A$, $(tx)^2=t^2=\alpha$ and $(tx)^{-1}\cdot y\cdot tx=y^{-1}$. Thus $G=\langle A,t\rangle=\langle A,tx\rangle$ for any $x\in A$. For this reason, we may assume that  $s=t$.

The following lemma is straightforward.
\begin{lemma}\label{lem::tA}
Let $g\in G$. Then each vertex  $x\in gH$ has exactly two neighbors outside $gH$,  namely $xt$ and $xt^{-1}$.
\end{lemma}

We consider the following two situations.
{\flushleft \bf Case A.}  $t^2=\alpha\in H$.
{\flushleft \bf Subcase A.1.}  $S\setminus\{t,t^{-1}\}\subseteq A$.

First suppose that $t^2=\alpha\not\in S$. Consider the vertices $1\in H$ and $t\in tH$. By Lemma \ref{lem::tA}, the only possible common neighbors of $1$ and $t$ are $t^{-1}$ and $t^2$, and $t^{-1}\in N(1)\cap N(t)$ if and only if $t^2\in N(1)\cap N(t)$. However, since $t^2\not\in S$, the vertices $1$ and $t^2$ cannot be adjacent. Thus $N(1)\cap N(t)=\emptyset$, and $a_1=0$. As $|S|\geq 3$, $S\setminus\{t,t^{-1}\}\neq \emptyset$. For any $x\in S\setminus\{t,t^{-1}\}=S\cap A$,  we have $\partial(x,t)=2$ because $tx$ and $t^{-1}x=(tx)^{-1}$ cannot be contained in $S$ due to  $x\neq t^2$. Note that $x\in H$ and $t\in tH$. Again by Lemma \ref{lem::tA}, we see that $N(x)\cap N(t)=\{1,xt\}$ when $xt^2\not\in S$, and $N(x)\cap N(t)=\{1,xt,t^2,xt^{-1}\}$ when $xt^2\in S$. Thus $c_2=2$ or $4$. If $c_2=2$, then $xt^2\not\in S$ for any $x\in S\setminus\{t,t^{-1}\}$.  Pick $x\in S\setminus\{t,t^{-1}\}$. Clearly, $x^2\neq 1$, since otherwise we must have $x=\alpha=t^2\not\in S$ because $\alpha$ is the unique involution of $G$, a contradiction. Consider the two vertices $1,x^2\in H$. Observe that $\partial(1,x^2)=2$ because   $a_1=0$ and $x\in N(1)\cap N(x^2)$. Also, $x^2\neq t^2$, since otherwise we have $\{x,t,t^{-1}\}\subseteq N(1)\cap N(x^2)$, contrary to $c_2=2$. Moreover, we assert that  $1$ and $x^2$ have no common neighbors outside $H$, since otherwise it follows from Lemma \ref{lem::tA} that $tx^2\in S$, contrary to the assumption that $S\setminus \{t,t^{-1}\}\subseteq A$. Therefore, there exists exactly one $y\in S\setminus\{t,t^{-1},x\}$ which is the common neighbor of $1$ and $x^2$. Note that $\partial(x,y)=2$. We claim that $y=x^{-1}$, since otherwise we obtain $c_2\geq 3$ by observing that  $\{1,xy,x^2\}\subseteq N(x)\cap N(y)$,  a contradiction. Thus $N(1)\cap N(x^2)=\{x,x^{-1}\}$, and $x^3\in S$. This implies that $x^{-2}\in N(x)\cap N(x^{-1})$, and so $x^4=1$ because $\{1,x^2\}\subseteq N(x)\cap N(x^{-1})$, $x^{-2}\neq 1$, $\partial(x,x^{-1})=2$ and $c_2=2$.  Then $o(x^2)=2$, and we have $x^2=\alpha=t^2$ because $\alpha$ is the unique involution of $G$. However, this is impossible by above arguments. If $c_2=4$, then $xt^2\in S$ for each $x\in S\setminus\{t,t^{-1}\}$. As $S\setminus\{t,t^{-1}\}\neq \emptyset$ and $t^2=\alpha\not\in S$, we conclude that $|S\setminus\{t,t^{-1}\}|\geq 2$. Consider the two vertices $1,t^2\in H$. Note that $\partial(1,t^2)=2$. Pick two distinct $x,y\in S\setminus\{t,t^{-1}\}=S\cap A$. We claim that $x^{-1}y=t^2=y^{-1}x$, since otherwise it follows from $\{x,xt^2,y,yt^{2},t,t^{-1}\}\subseteq N(1)\cap N(t^2)$ that $|N(1)\cap N(t^2)|\geq 6$, contrary to $c_2=4$. Hence,  by the arbitrariness of $x,y\in S\setminus\{t,t^{-1}\}$, we conclude that $S\setminus\{t,t^{-1}\}=\{x,xt^2\}$. Then $x^{-1}=x$ or $x^{-1}=xt^2$ because $x^{-1}\in S\setminus\{t,t^{-1}\}$. If $x^{-1}=x$, then $o(x)=2$, and so $x=\alpha=t^2\not\in S$, a contradiction. If $x^{-1}=xt^2$, then $x^2=t^2=\alpha$ and $S\setminus\{t,t^{-1}\}=\{x,x^{-1}=x^3\}$. Therefore, $\Gamma=\mathrm{Cay}(G,\{x,x^{-1}=x^3,t,t^{-1}\})\cong K_{4,4}$.

Now suppose that $t^2=\alpha\in S$. Consider the vertices $1\in H$ and $t\in tH$. By Lemma \ref{lem::tA}, it is easy to see that $N(1)\cap N(t)=\{t^2,t^{-1}\}$. As $1$ and $t$ are adjacent, we have $a_1=2$. Clearly,  $|S|\geq 4$, since $\{t,t^{-1},t^2=\alpha\}$ cannot generate $G$ due to $n>1$. Pick $x\in S\setminus\{t,t^{-1},t^2=\alpha\}$. Then $x\in H\leq A$ and $\partial(x,t)=2$. Again by Lemma \ref{lem::tA}, the two neighbors of $x$ outside $H$ are $xt=tx^{-1}\in tH$ and $xt^{-1}=t^{-1}x^{-1}\in t^{-1}H=tH$, and the two neighbors of $t$ outside $tH$ are $1\in H$ and $t^2\in H$. Thus $N(x)\cap N(t)=\{1,tx^{-1}\}$ or $\{1,tx^{-1}, t^2, t^{-1}x^{-1}\}$. We assert that the later case cannot occur. In fact, if $t^2\in N(x)$, then $\{x,xt^2,t,t^{-1}\}\subseteq N(1)\cap N(t^2)$, contrary to $a_1=2$ because $1$ and $t^2$ are adjacent. Thus $c_2=2$, and every pair of vertices with distance at most $2$ have exactly two common neighbors. Note that $x^2\neq 1$ because $x\neq t^2=\alpha$. Since  $x\in N(1)\cap N(x^2)$, we have $\partial(1,x^2)\leq 2$, and $|N(1)\cap N(x^2)|=2$.  Let $y$ be the remaining common neighbor of $1$ and $x^2$ other than $x$. Note that $x^2\neq t^2$ because $x\not\in N(1)\cap N(t^2)=\{t,t^{-1}\}$. According to Lemma \ref{lem::tA}, this implies that $y\in S\setminus\{t,t^{-1}\}\subseteq A$. Then from $\partial(x,y)\leq 2$, $\{1,x^2,xy\}\subseteq N(x)\cap N(y)$ and $a_1=c_2=2$,  we can deduce that  $y=x^{-1}$,  $N(x)\cap N(x^{-1})=\{1,x^2\}$ and $x^3\in S$. It follows that $x^{-2}$ is also a common neighbor of $x$ and $x^{-1}$, and so we must have $x^{-2}=x^2$, i.e., $o(x^2)=2$. Therefore,  $x^2=\alpha=t^2$, which is impossible by above arguments.

{\flushleft \bf Subcase A.2.}  $(S\setminus\{t,t^{-1}\})\cap  tA\neq \emptyset$.

Pick $tx\in (S\setminus\{t,t^{-1}\})\cap  tA$. Then $x\in A$, $x\neq t^2=\alpha$ and $t^2=txtx\in H$. We claim that $t^2=\alpha\not\in S$. Indeed, if $t^2\in S$, by Lemma \ref{lem::tA}, we have $N(1)\cap N(t)=\{t^2,t^{-1}\}$, which leads to $a_1=2$ because $1$ and $t$ are adjacent. On the other hand, since $\{t,t^{-1},tx\} \subseteq N(1)\cap N(t^2)$, we get $a_1\geq 3$, a contradiction. Therefore, we have $t^2\not\in S$, $N(1)\cap N(t)=\emptyset$ and $a_1=0$. Note that $t^{-1}x=(tx)^{-1}\in S\setminus\{t,t^{-1}\}$ and $t^{-1}xt^{-1}x=t^2$. We have $\{t,t^{-1},tx,t^{-1}x\}\subseteq N(1)\cap N(t^2)$, and so  $c_2\geq 4$ because $\partial(1,t^2)=2$. Now we shall prove that  $S\cap A=\emptyset$. By contradiction, assume that $S\cap A\neq \emptyset$. Pick $y\in S\cap A$. Clearly,  $y\neq t^2$, and so $\partial(y,t)=2$ by Lemma \ref{lem::tA}. Note that the only possible common neighbors of $y\in H$ and $t\in tH$ are $1$, $t^2$, $yt=ty^{-1}$ and $yt^{-1}=t^{-1}y^{-1}$. From $c_2\geq 4$ we can deduce that $N(y)\cap N(t)=\{1,t^2,yt,yt^{-1}\}$ and $c_2=4$. But then we have $\{y,tx,t^{-1}x,t,t^{-1}\}\subseteq N(1)\cap N(t^2)$, which is impossible because $c_2=4$. Hence, $S\cap A=\emptyset$. Note that $\partial(t,tx)=2$ because $t,tx\in S$ and $a_1=0$.  Since $c_2\geq 4$ and the only possible common neighbors of $t\in tH$ and $tx\in H$ are $1$, $t^2$, $txt^{-1}=x^{-1}$ and $txt=t^2x^{-1}$, we have $N(t,tx)=\{1,t^2,x^{-1},t^2x^{-1}\}$ and $c_2=4$. This implies that $t$ is adjacent to $x^{-1}$, and therefore, $\{tx^{-1},t^{-1}x^{-1}=(tx^{-1})^{-1}\}\subseteq (S\setminus\{t,t^{-1}\})\cap  tA$. Then  $\{t,t^{-1},tx,t^{-1}x,tx^{-1},t^{-1}x^{-1}\}\subseteq N(1)\cap N(t^2)$, and it follows from $c_2=4$ and $x\neq t^2=\alpha$ that  $tx=t^{-1}x^{-1}$, i.e., $x^2=t^2=\alpha$.  Next we assert that $(S\setminus\{t,t^{-1}\})\cap  tA=\{tx,t^{-1}x=tx^{-1}\}$. If not, there exists some $y\in A\setminus\{x,x^{-1}\}$ such that $ty\in (S\setminus\{t,t^{-1}\})\cap  tA$. By above arguments, we have $y^2=t^2=\alpha$, and  $xyxy=x^2y^2=\alpha^2=1$. Thus $o(xy)=1$ or $o(xy)=2$. However, both of them are impossible because the first one implies $y=x^{-1}$ and the second one implies  $xy=\alpha=x^2$, i.e., $y=x$. As $S\cap A=\emptyset$, we may conclude that $S=\{t,t^{-1},tx,t^{-1}x=tx^{-1}\}$ with $o(x)=4$, and  $\Gamma=\mathrm{Cay}(G,S)\cong K_{4,4}$.

{\flushleft \bf Case B.}  $t^2=\alpha\not\in H$.

In this situation, we assert that $S\setminus\{t,t^{-1}\}\subseteq A$. Indeed, if there exists some $tx\in (S\setminus\{t,t^{-1}\})\cap  tA$, then $t^2=txtx\in H$, contrary to our assumption. Since $t^2\not\in S$, as above, we have $N(1)\cap N(t)=\emptyset$ and $a_1=0$. By Lemma \ref{lem::tA}, the two neighbors of $1\in H$ (resp. $t^2\in t^2H$) outside $H$ (resp. $t^2H$) are $t\in tH$ and $t^{-1}\in t^{-1}H=t^3H$. Also note that $t^iH\neq t^jH$ for $0\leq i\neq j\leq 3$. Thus we have $N(1)\cap N(t^2)=\{t,t^{-1}\}$, and so $c_2=2$. Recall that $|S|\geq 3$. Pick $x\in S\setminus\{t,t^{-1}\}=S\cap A$. Clearly, $x^2 \neq 1$, since otherwise we can deduce that $x=\alpha=t^2\not\in S$, a contradiction. Then $\partial(1,x^2)=2$ because $x\in N(1)\cap N(x^2)$ and $a_1=0$. Let $y$ be the remaining common neighbor of $1$ and $x^2$ other than $x$. As above, we conclude that $y=x^{-1}$ (because $c_2=2$), $\partial(x,y)=2$, and $xy$ is also a common neighbor of $x$ and $y$. Then $x^{3}\in S$, and it follows that $x^{-2}=x^2$ because $x^{-2}$ is also a common neighbor of $x$ and $x^{-1}$. Therefore, we have $o(x^2)=2$, and  $t^2=\alpha=x^2\in H$, contrary to the assumption.

\subsection{The case $s\in A$}

In this part, the main method used in the proof is similar as that of \cite{MS20}.

First we claim that $o(s)>2$. Indeed, if $o(s)=2$, then $s=\alpha$. Since $S\cap tA\neq \emptyset$, we can take  $tx\in S\cap tA$ such that $txtx=t^2=\alpha=s$, contrary to the fact that  $S$ is  minimal with respect to $s$.  Furthermore, it is easy to see that $G$ is the disjoint union of the left cosets $s^iH$, where $0\leq i\leq [G:H]-1$, and that $o(s)$ is a multiple of $[G:H]$.
Also recall that for each $i$ the subgraph of $\Gamma$ induced by $s^iH$ is isomorphic to $\Gamma'$.

\begin{lemma}\label{lem::A}
The following statements hold.
\begin{enumerate}[(i)]
\item For each $h\in H$ and for each $0\leq i\leq [G:H]-1$, the vertex $s^ih\in s^iH$ has exactly two neighbors outside $s^iH$, namely $s^{i-1}h\in s^{i-1}H$ and $s^{i+1}h\in s^{i+1}H$ (here $s^{i-1}h\neq s^{i+1}h$ due to $o(s)>2$). Moreover, $s^{i-1}H=s^{i+1}H$ if and only if  $[G:H]=2$ and $o(s)\geq 4$.
\item $c_2\geq 2$.
\end{enumerate}
\end{lemma}
\begin{proof}
Since $s\in A$, and $o(s)>2$ is a multiple of $[G:H]$,  the statement in (i) is obvious. For (ii), we take $h\in S\setminus \langle s \rangle$.  Consider the vertices $1$ and $sh$. Clearly, $1$ and $sh$ are not adjacent because $S$ is minimal with respect to $s$ and $h\neq s^{-2}$. Then from $\{s,h\}\subseteq N(1)\cap N(sh)$ we obtain $\partial(1,sh)=2$ and $c_2\geq 2$.
\end{proof}

The remaining part of the proof consists of a series of claims.

\begin{claim}\label{claim::1}
$[G:H]=2$ and $o(s)\geq 4$.
\end{claim}
\begin{proof}
First assume that $[G:H]\geq 5$. Then $o(s)\geq [G:H]\geq 5$. Consider the vertices $1\in H$ and $s^2\in s^2H\neq H$. By Lemma \ref{lem::A}, the two neighbors of $1$ outside $H$ are $s\in sH$ and $s^{-1}\in s^{-1}H$, and the two neighbors of $s^2$ outside $s^2H$ are $s\in sH$ and $s^3\in s^3H$. Since $s^{-1}\neq s^3$, we have $N(1)\cap N(s^2)=\{s\}$, which is impossible because $\partial(1,s^2)=2$ and $c_2\geq 2$ by Lemma \ref{lem::A}.

Next assume that $[G:H]=4$. We have $o(s)\geq 4$. If $o(s)\geq 5$, as above, we obtain a contradiction. Thus $o(s)=4$, and $s^2=\alpha$. Recall that $S\cap tA\neq\emptyset$. Take $tx\in S\cap tA$. Then $s^2=\alpha=t^2=txtx\in H$, which implies that $s^2H=H$. Therefore,  $[G:H]=2$, a contradiction.

Now assume that $[G:H]=3$. Recall that $o(s)$ is a multiple of $[G:H]$. We consider the following two cases.
{\flushleft \bf Case A.}  $o(s)\geq 6$.

As above, consider the vertices $1\in H$ and $s^2\in s^2H$. Note that $\partial(1,s^2)=2$ because $s\in N(1)\cap N(s^2)$ and $s^2\not\in  H$ due to $[G:H]=3$. Since $c_2\geq 2$, by Lemma \ref{lem::A}, $1$ and $s^2$ have at least one more common neighbor, which can only be $s^{-1}\in s^2H$ or $s^3\in H$. In both cases, we get $s^3\in S$, and hence $\{s^{-1},s^3\}\subseteq N(1)\cap N(s^2)$. This implies that $c_2=3$. Now consider the vertices $1\in H$ and $s^4\in s^4H=sH$. Clearly, $1$ and $s^4$ are not adjacent, and so $\partial(1,s^4)=2$ because $\{s,s^3\}\in N(1)\cap N(s^4)$. Again by Lemma \ref{lem::A}, the remaining common neighbor of $1$ and $s^4$ other than  $s,s^3$ can only be  $s^{-1}=s^5$. Hence, $o(s)=6$. Since $S\cap tA\neq \emptyset$, we can take  $x\in A$ such that $tx\in S$. Clearly, $1\in H$ and $stx\in sH$ are not adjacent because $S$ is minimal with respect to $s$. Observe that $stx=txs^{-1}$, we have $\{s,tx\}\subseteq N(1)\cap N(stx)$, and $\partial(1,stx)=2$. As $c_2=3$, the vertices $1$ and $stx$ have another common neighbor, which can only be $s^{-1}=stxs^{-1}$ by Lemma \ref{lem::A}. In this situation, we obtain $s=t^{-1}x$, which is impossible.

{\flushleft \bf Case B.}  $o(s)= 3$.

Recall that $G$ is the disjoint union of $H$, $sH$ and $s^2H$, and that the subgraph  $\Gamma[s^iH]$ of $\Gamma$ induced by $s^iH$ is isomorphic to $\Gamma'$ for $0\leq i\leq 2$. We claim that $G\cong K_3\square  \Gamma'$. In fact, for every pair of vertices $x,y\in H$, the vertices $s^ix,s^iy\in s^iH$ are adjacent in $\Gamma[s^iH]$ if and only if $x,y$ are adjacent in $\Gamma[H]=\Gamma'$. Moreover, by Lemma \ref{lem::A}, each vertex $x\in s^iH$ ($0\leq i\leq 2$) has exactly two neighbors outside $s^iH$, namely $xs=sx\in s^{i+1}H$ and $xs^{-1}=s^{-1}x\in s^{i-1}H$, or $xs=s^{-1}x\in s^{i-1}H$ and $xs^{-1}=sx\in s^{i+1}H$. Thus we conclude that $\Gamma\cong K_3\square  \Gamma'$. As a Cartesian product of graphs can be factorized uniquely as a product of prime factors, Lemma \ref{lem::CP} implies that $\Gamma$ is isomorphic to the Hamming graph $H(d,3)$ for some positive integer $d$. However, this is impossible because $\Gamma$ is of even  order $4n$.

Therefore, we  have $[G:H]=2$, and so $o(s)\geq 4$ because  $o(s)>2$ is a multiple of  $[G:H]$.
\end{proof}

According to Claim \ref{claim::1}, $G$ is the disjoint union of the left cosets $H$ and $sH$. Then we can obtain the following result.

\begin{claim}\label{claim::2}
We have
\begin{enumerate}[(i)]
\item $a_1\in\{0,2\}$, and $a_1=2$ if and only if $s^2\in S$;
\item $c_2\in \{2,4\}$, and if $o(s)\geq 6$ then $c_2= 4$.
\end{enumerate}
\end{claim}
\begin{proof} (i) We consider the vertices $1\in H$ and $s\in sH$. By Lemma \ref{lem::A}, the common neighbors of $1$ and $s$ can only be $s^{-1}\in s^{-1}H=sH$ and $s^2\in s^2H=H$. Note that $s^{-1}\in N(1)\cap N(s)$ if and only if $s^2\in N(1)\cap N(s)$. Since $s^2\neq s^{-1}$ due to $o(s)\geq 4$, we have  $a_1=0$ or $2$, and $a_1=2$ if and only if $s^2\in S$.

\vspace*{2mm}

\noindent
(ii) Pick $tx\in S\cap tA$. Consider the vertices $1\in H$ and $stx\in sH$. Clearly, they are not adjacent because $S$ is minimal with respect to $s$.  As  $\{s,tx\}\in N(1)\cap N(stx)$, we have $\partial(1,stx)=2$. By Lemma \ref{lem::A}, the only other possible common neighbors of $1$ and $stx$ are $s^{-1}\in s^{-1}H=sH$ and $stxs^{-1}=s^2tx\in s^2H=H$. Furthermore, we see that $s^{-1}\in N(1)\cap N(stx)$ if and only if $s^2tx\in S$, which is the case if and only if $s^2tx\in N(1)\cap N(stx)$. Clearly, $s$, $s^{-1}$, $tx$ and $s^2tx$ are pairwise distinct. Therefore, we have $N(1)\cap N(stx)=\{s,tx\}$ or $\{s,s^{-1},tx,s^2tx\}$, and so $c_2=2$ or $4$.  Suppose that $o(s)\geq 6$ and  $c_2=2$. Consider the vertices $1\in H$ and $s^2\in s^2H=H$. Clearly, $\partial(1,s^2)\leq 2$ because $s\in N(1)\cap N(s^2)$. Since  $a_1=2$ if and only if $s^2\in S$, we assert that $1$ and $s^2$ has exactly two common neighbors. As $o(s)\geq 6$, by Lemma \ref{lem::A}, we see that $s$ is the unique common neighbor of $1$ and $s^2$ in $sH$. Thus the remaining common neighbor of $1$ and $s^2$ other than $s$, say $y$, must be contained in $H$.  Then $\partial(s,y)\leq 2$, and  $1$ and $s^2$ are the only common neighbors of $s$ and $y$. Furthermore, we see that $sy=ys$ or $sy=ys^{-1}$ is also a common neighbor of $s$ and $y$. Thus we  have $sy\in\{1,s^2\}$, and hence $sy=1$ because $s\neq y$. However, this implies that $y=s^{-1}\in s^{-1}H=sH$, a contradiction. Therefore, if $o(s)\geq 6$ then $c_2=4$.
\end{proof}

\begin{claim}\label{claim::3}
 $o(s)=4$.
\end{claim}
\begin{proof}
By the way of contradiction, we assume that $o(s)\geq 6$. By Claim \ref{claim::2}, we have $a_1\in\{0,2\}$ and $c_2=4$.

First suppose that $a_1=2$. Then $s^2\in S$ by Claim \ref{claim::2}. We assert that $\{s^2,s^4,\ldots,s^{o(s)-2}\}$ $\subseteq S$. In fact, assume that $\{s^2,s^4,\ldots,s^{2i}\}\subseteq S$ for some $1\leq i<\frac{o(s)-2}{2}$. Consider the vertices $1\in H$ and $s^{2i+1}\in s^{2i+1}H=sH$. Clearly, $\partial(1,s^{2i+1})=2$ because $s^{2i+1}\neq s^{-1}$ and $S$ is minimal with respect to $s$. Then $1$ and $s^{2i+1}$ have exactly four common neighbors, which can only be $s$, $s^{-1}$, $s^{2i}$ and $s^{2i+2}$ by Lemma \ref{lem::A}. Thus $s^{2i+2}\in S$, and our assertion follows. Then we see that  $\{s,s^4,\ldots,s^{o(s)-2}\}\subseteq N(1)\cap N(s^2)$, which leads to $o(s)=6$ because $a_1=2$.
Pick $tx\in S\cap tA$. Note that $\partial(1,stx)=2$. Then $1\in H$ and $stx\in sH$ have exactly four common neighbors, which can only be $s$, $s^{-1}$, $stxs=tx$ and $stxs^{-1}=s^2tx$ by Lemma \ref{lem::A}. It follows that $s^2tx\in S$, and  so $s^2tx\in N(1)\cap N(s^2)$. Thus $1$ and $s^2$ have at least three common neighbors, which is impossible due to $a_1=2$.

Now suppose that $a_1=0$.  Then $s^2\not\in S$ by Claim \ref{claim::2}. Pick $y\in S\setminus\langle s \rangle$. Note that $\partial(1,sy)=2$. We assert that $\{y,s^2y,s^4y,\ldots,s^{o(s)-2}y\}\subseteq S$. Indeed, assume that $\{y,s^2y,s^4y,\ldots,s^{2i}y\}\subseteq S$ for some $0\leq i<\frac{o(s)-2}{2}$. Consider the vertices $1\in H$ and $s^{2i+1}y\in s^{2i+1}H=sH$. Clearly, $\partial(1,s^{2i+1}y)=2$ because $s^{2i+1}y\neq s^{-1}$ due to $y\not\in \langle s \rangle$ and $S$ is minimal with respect to $s$. By Lemma \ref{lem::A}, we see that $1$ and $s^{2i+1}y$ have exactly four common neighbors, namely $s$, $s^{-1}$, $s^{2i}y$ and $s^{2i+2}y$. Thus $s^{2i+2}y\in S$ and the assertion follows. Then  $\{s,y,s^2y,\ldots,s^{o(s)-2}y\}\subseteq N(1)\cap N(s^2)$, which implies that $o(s)=6$ because $\partial(1,s^2)=2$ and $c_2=4$.  Since $s$, $y$, $s^2y$ and $s^4y$ are the four distinct common neighbors of $1$ and $s^2$, we conclude that $S\setminus \langle s\rangle=\{y,s^2y,s^4y\}$ by the above arguments and the arbitrariness of $y$. Note that $s^2,s^4=(s^2)^{-1}\not\in S$ because $a_1=0$. Moreover, we assert that $s^3\not\in S$, since otherwise $S$ cannot  be minimal with respect to $s$ because   $s=s^{3} \cdot y \cdot (s^2y)^{-1}$ and  $\{s^3,y,s^2y\}\subseteq S$, contrary to our assumption. Therefore,  $S=\{s,s^{-1}=s^5,y,s^2y,s^4y\}$.  As $\langle S\rangle=G$, we have $y\not\in A$, and so $y=tx$ for some $x\in A$. Then it follows from $y^{-1}=t^{-1}x\in S$ that $t^{-1}x=s^2y=s^2tx=ts^4x$ or $t^{-1}x=s^4y=s^4tx=ts^2x$, i.e., $t^2=s^4$ or $t^2=s^2$. Therefore, as $t^4=1$, we have $s^8=1$ or $s^4=1$, which is impossible because   $o(s)=6$.
\end{proof}

\begin{claim}\label{claim::4}
$\Gamma$ is isomorphic to $K_{4,4}$.
\end{claim}
\begin{proof}
By Claim \ref{claim::3}, we have $o(s)=4$, and so  $s^2=\alpha=t^2$. Pick $tx\in S\cap tA$. Then $t^{-1}x=(tx)^{-1}\in S$, and we see that $s^2=t^2=txtx=t^{-1}xt^{-1}x$. This implies that $\{s,s^{-1},tx,t^{-1}x\}\subseteq N(1)\cap N(s^2)$. Thus $s^2\not \in S$ (i.e., $a_1=0$) and $c_2=4$ by Claim \ref{claim::2}. By the arbitrariness of $tx\in S\cap tA$, we conclude that $S\cap tA=\{tx,t^{-1}x\}$, and so $S$ is also minimal with respect to $tx$.  According to what we have proved in the previous subsection, we assert that $\Gamma$ can only be isomorphic to $K_{4,4}$.  Note that $K_{4,4}\cong \mathrm{Cay}(G,\{s,s^{-1}=s^3,t,t^{-1}\})$. The result follows.
\end{proof}

Concluding the above results, we finish the proof of Theorem \ref{thm::main}.

\section{Concluding remarks}

In this paper, we prove that $K_{4,4}$ is the unique distance-regular Cayley graph of generalized dicyclic groups under the condition that the corresponding connection set is minimal. 

For a distance-regular graph  $\Gamma$ with diameter $d$, the \textit{$i$-th distance graph} $\Gamma_i$ is defined as the graph with vertex set  $V(\Gamma)$ in which two vertices are adjacent if and only if they are at distance $i$ in $\Gamma$. We say that $\Gamma$ is \textit{primitive} if $\Gamma_i$ is connected for all $1\leq i\leq d$, and \textit{imprimitive} otherwise.  

Let $G$ be a generalized dicyclic group of order $4n$. Assume that $\Gamma=\mathrm{Cay}(G,S)$ is a primitive distance-regular Cayley graph of $G$, and that $\mathcal{D}=\mathcal{D}_\mathbb{Z}(G,S)$ is the distance module of $\Gamma$ (see \cite{MP03} for the definition). According to \cite[Proposition 3.6(i)]{MP03}, $\mathcal{D}$ is a primitive Schur ring over $G$, and so must be a trivial Schur ring by \cite[Theorem 4]{S57}. Then it follows from \cite[Proposition 3.6(ii)]{MP03} that $\Gamma$ must be isomorphic to the complete graph $K_{4n}$. Therefore, in order to characterize distance-regular Cayley graphs of generalized dicyclic groups, it suffices to consider those that are imprimitive. Also note that an imprimitive distance-regular graph of valency at least $3$ is either bipartite, antipodal, or both \cite[Theorem 4.2.1]{BCN89}. In future, we will consider to  classify the distance-regular Cayley graphs of generalized dicyclic groups that are  bipartite or antipodal.

\section*{Declaration of competing interest}

The authors declare that they have no known competing financial interests or personal relationships that could have
appeared to influence the work reported in this paper.

\section*{Acknowledgements}

The authors are grateful to the anonymous referees for their useful and constructive comments, which have considerably improved the presentation of this paper. X. Huang is supported by National Natural Science Foundation of China (Grant No. 11901540). K. C. Das is supported by National Research Foundation funded by the Korean government (Grant No. 2021R1F1A1050646).

\end{document}